\documentclass[11pt]{article}

\usepackage[margin=1in]{geometry}
\usepackage{amsmath,amssymb,amsthm}
\usepackage{mathtools}
\usepackage{booktabs}
\usepackage{array}
\usepackage{float}
\usepackage{enumitem}
\usepackage[colorlinks=true,linkcolor=blue,citecolor=blue,urlcolor=blue]{hyperref}

\newtheorem{theorem}{Theorem}
\newtheorem{lemma}[theorem]{Lemma}

\theoremstyle{definition}
\newtheorem{definition}[theorem]{Definition}
\newtheorem{remark}[theorem]{Remark}
\newtheorem{example}[theorem]{Example}

\DeclareMathOperator{\Li}{Li}
\DeclareMathOperator{\re}{Re}
\newcommand{\Q}{\mathbb{Q}}

\title{${}_5F_4$ evaluations and a family of $\pi^2+\log^2$ identities}

\author{Cetin Hakimoglu-Brown}
\date{}

\begin{document}

\maketitle

\begin{abstract}
We evaluate the series $\sum_{n\ge1} z^n\big/\!\big(n^2\binom{4n}{n}\big)$, equal to
$-\tfrac{z}{4}\,{}_5F_4\!\left(1,1,1,\tfrac43,\tfrac53;\tfrac54,\tfrac32,\tfrac74,2;\tfrac{27z}{256}\right)$,
in closed form at an infinite family of algebraic points indexed by a rational angle
$\theta=j\pi/N$. Each value equals $c\,\pi^2$ plus a universal rational quadratic form in three
logarithms, with $c=-\tfrac13\left(1-\tfrac{2j}{N}\right)^2$. This is the quartic-base case reached
but not evaluated by D'Aurizio and Di Trani. The proof is self-contained: an exact integer factor
relating two weights, followed by Landen's identity, reduces the integral to a sum of squared
logarithms. \end{abstract}

\section{Introduction}

For coprime positive integers $p,q$ set $s=p+q$ and consider the integral
\begin{equation}\label{eq:base-integral}
  \mathcal{I}_{p,q}(z)\;=\;\int_0^1 \frac{\ln\!\big(1-z\,x^p(1-x)^q\big)}{x}\,dx .
\end{equation}
Expanding the logarithm and integrating term by term gives
\begin{equation}\label{eq:series}
  \mathcal{I}_{p,q}(z)
  = -\sum_{n\ge1}\frac{(pn-1)!\,(qn)!}{(sn)!}\,\frac{z^n}{n},
\end{equation}
a series convergent for $|z|<z_c$, where $z_c=s^s/(p^pq^q)$ is the reciprocal of the maximum of
$x^p(1-x)^q$ on $[0,1]$.

The case $p=1,q=2$ (base $x(1-x)^2$, the reciprocal central binomial $\binom{3n}{n}$) was studied by
D'Aurizio and Di Trani~\cite{DaurizioDiTrani}, who obtained closed forms of the type
$c\pi^2+\text{(logarithms)}$; related $\binom{3n}{n}$ series appear in Batir~\cite{Batir} and
Adegoke, Frontczak and Goy~\cite{AFG}, and reciprocal central binomial series more broadly are
treated in~\cite{Sprugnoli,Lehmer,Kaushik}. The dilogarithmic and Bloch-group background we use is
classical~\cite{Lewin,Zagier,LewinStructural,Maximon}; the beta-integral technique underlying the
present construction is developed in~\cite{HBanalytic}. In the final section of~\cite{DaurizioDiTrani} the authors
turn to the quartic base $x(1-x)^3$ --- the reciprocal central binomial $\binom{4n}{n}$, for which
\eqref{eq:base-integral} is
\begin{equation}\label{eq:5F4}
  \mathcal{I}_{1,3}(z)
  = -\frac{z}{4}\;{}_5F_4\!\left(\begin{matrix}1,1,1,\tfrac43,\tfrac53\\[2pt]
      \tfrac54,\tfrac32,\tfrac74,2\end{matrix}\;\middle|\;\frac{27z}{256}\right)
  = -\sum_{n\ge1}\frac{z^n}{n^2\binom{4n}{n}} ,
\end{equation}
In their final section they carry out this reduction, obtaining for the mirror base $x^3(1-x)$ a sum
of squared logarithms over the four fiber roots --- the analogue of \eqref{eq:landen-form} below.
They then impose that the fiber vanish at a single \emph{real} point $-m$, which gives
$z=-1/(m^3+m^4)$, and stop there: in their words, ``the final expression does not simplify as nicely
as before, involving a fair amount of cube roots.'' No evaluation of the ${}_5F_4$ is given.

The present paper takes the reduction as its starting point and supplies what is missing. The cube
roots are an artefact of forcing a real root; forcing instead a complex-conjugate pair at a
prescribed rational angle leaves a real quadratic cofactor, and the evaluation closes. Throughout, write $\mathcal{I}(z):=\mathcal{I}_{1,3}(z)$. Our results are:

\begin{itemize}
\item \textbf{A uniform family (Theorem~\ref{thm:main}).} For each rational angle $\theta=j\pi/N$
that admits a branch (real positive $r$ with $S^2-4P>0$ and $|z|<z_c$) there is an algebraic point
$z=z(j,N)$ at which
\[
  \mathcal{I}(z)= -\tfrac13\Big(1-\tfrac{2j}{N}\Big)^2\pi^2 + Q(\ln a,\ln|b|,\ln r),
\]
where $a,b,r$ are explicit real algebraic numbers and $Q$ is a fixed rational quadratic form,
\emph{independent of $(j,N)$}. Every angle we have examined admits such a branch.
\item \textbf{A self-contained proof (Section~\ref{sec:proof}).} The coefficient of $\pi^2$ and the
quadratic form are derived, not merely verified: an exact integer factor relating two weights,
followed by Landen's identity, reduces $\mathcal{I}(z)$ to a sum of four squared logarithms which
collapse onto three generators via the defining fiber relation.
\item \textbf{A boundary theorem (Section~\ref{sec:boundary}).} A parity count of the fiber's real
roots shows that the forced pair is the only complex pair --- so that purity is automatic --- exactly
for the ${}_4F_3$ and ${}_5F_4$ bases, while every higher order retains an unconstrained conjugate
pair. In this sense ${}_5F_4$ is the maximal such order. Whether the leftover pair can nonetheless
land at a rational angle for isolated higher-order parameters is not settled here; an exhaustive
search in the quintic case found no such member.
\end{itemize}

\section{The fiber and the matching construction}\label{sec:setup}

Throughout, $\log$ denotes the principal logarithm, with $\arg$ the principal argument taking values
in $(-\pi,\pi]$, and $\Li_2$ the principal branch of the dilogarithm, defined on
$\mathbb{C}\setminus[1,\infty)$ by $\Li_2(w)=-\int_0^w t^{-1}\log(1-t)\,dt$ along any path avoiding
the cut. For an admissible member these branches are never in question. By Remark~\ref{rem:signs} the real
roots satisfy $a>1$ and $b<0$, so $1/a\in(0,1)$ and $1/b<0$, while the remaining two roots are
nonreal; hence $1/\rho_k\notin[1,\infty)$ for every $k$. Likewise $1-1/a\in(0,1)$ and $1-1/b>1$ are
positive and the conjugate pair is nonreal, so $1-1/\rho_k\notin(-\infty,0]$. Every logarithm and
dilogarithm below is therefore evaluated away from its cut, and the identities used hold in their
displayed form with no additional multiple of $2\pi i$.

Since $1-z\,x(1-x)^3$ has constant term $1$ at $x=0$, it factors as
$\prod_{k=1}^4\big(1-x/\rho_k\big)$ where $\rho_1,\dots,\rho_4$ are the roots of the
\emph{fiber polynomial}
\begin{equation}\label{eq:fiber}
  x^4-3x^3+3x^2-x+\tfrac1z=0,\qquad\text{equivalently}\qquad x(1-x)^3=\tfrac1z .
\end{equation}
Consequently
\begin{equation}\label{eq:dilog}
  \mathcal{I}(z)=\sum_{k=1}^4\int_0^1\frac{\ln(1-x/\rho_k)}{x}\,dx
  = -\sum_{k=1}^4\Li_2\!\Big(\frac1{\rho_k}\Big).
\end{equation}

We select $z$ so that two of the four roots form a complex-conjugate pair at a prescribed rational
angle. Fix $\theta=j\pi/N$ with $\gcd(j,N)=1$, put $\kappa=2\cos\theta$, and require the factor
$x^2-\kappa r\,x+r^2$ (whose roots are $r e^{\pm i\theta}$) to divide \eqref{eq:fiber}. Matching
coefficients yields, with $S=a+b$ and $P=ab$ for the remaining quadratic factor $x^2-Sx+P$,
\begin{equation}\label{eq:matching}
  S=3-\kappa r,\qquad P=3-r^2-3\kappa r+\kappa^2r^2,\qquad
  (\kappa^3-2\kappa)\,r^3+3(1-\kappa^2)\,r^2+3\kappa r-1=0,
\end{equation}
and then
\begin{equation}\label{eq:z}
  z=\frac1{P\,r^2}=\frac1{a\,(1-a)^3}.
\end{equation}
The last equality holds because $a$ is itself a root of \eqref{eq:fiber}.

\begin{definition}\label{def:admissible}
A triple $(\theta,r,z)$ arising from \eqref{eq:matching}--\eqref{eq:z}, with $r$ a positive real root
of the matching cubic, is called \emph{admissible} if $S^2-4P>0$ and $|z|<z_c=256/27$. Equivalently:
the quadratic cofactor $x^2-Sx+P$ has two real roots $a,b$, the forced pair $re^{\pm i\theta}$ is the
only nonreal pair of the fiber, and the series \eqref{eq:series} converges at $z$.
\end{definition}

The matching cubic may have more than one positive root; only an admissible branch is meant
throughout, and Lemma~\ref{lem:modulus} below identifies it in closed form.

\begin{remark}
Equation~\eqref{eq:z} is the same one-parameter substitution $z=1/(t(1-t)^3)$ that opens the
analysis; the pure family is exactly the locus of real roots $t=a$ for which the forced angle
$\theta$ is a rational multiple of $\pi$. D'Aurizio and Di Trani forced a single \emph{real} root,
leaving an irreducible cubic cofactor; forcing a \emph{conjugate pair} instead leaves the real
quadratic cofactor $(x-a)(x-b)$ and is what makes the reduction close.
\end{remark}

\section{The closed form and its proof}\label{sec:proof}

The proof avoids the five-term (Abel) relation entirely, and with it the root-difference
logarithms $\ln|\rho_i-\rho_j|$ that would otherwise obstruct a three-generator closed form. The
mechanism is a change of weight.

\subsection{An exact integer factor}

\begin{lemma}\label{lem:factor}
For $x^p(1-x)^q$ with $z$ in the domain of convergence,
\[
  \int_0^1\frac{\ln\!\big(1-z\,x^p(1-x)^q\big)}{1-x}\,dx
  \;=\;\frac{p}{q}\int_0^1\frac{\ln\!\big(1-z\,x^p(1-x)^q\big)}{x}\,dx .
\]
\end{lemma}

\begin{proof}
Expand $\ln(1-z\,x^p(1-x)^q)=-\sum_{n\ge1}\tfrac{z^n}{n}x^{pn}(1-x)^{qn}$. The $n$-th coefficients of
the two integrals are $\tfrac1n B(pn,qn+1)$ and $\tfrac1n B(pn+1,qn)$ respectively, and
\[
  \frac{B(pn+1,qn)}{B(pn,qn+1)}
  =\frac{(pn)!\,(qn-1)!}{(pn-1)!\,(qn)!}
  =\frac{pn}{qn}=\frac{p}{q}
\]
for every $n$. For $|z|<z_c$ the series $\sum_n \tfrac{z^n}{n}x^{pn}(1-x)^{qn}$ converges absolutely
and uniformly on $[0,1]$, so termwise integration is legitimate; summing gives the claim, and the
identity extends to the full domain by analytic continuation. \qedhere
\end{proof}

For $x(1-x)^3$ we have $p/q=1/3$, hence
\begin{equation}\label{eq:sym-weight}
  \int_0^1\frac{\ln\!\big(1-z\,x(1-x)^3\big)}{x(1-x)}\,dx
  =\Big(1+\tfrac13\Big)\mathcal{I}(z)=\frac{4}{3}\,\mathcal{I}(z),
  \qquad\text{i.e.}\qquad
  \mathcal{I}(z)=\frac{q}{p+q}\!\int_0^1\!\frac{\ln(\cdots)}{x(1-x)}\,dx .
\end{equation}

\subsection{The symmetric weight reduces to squared logarithms}

\begin{lemma}\label{lem:landen}
For any $z$ in the domain of convergence,
\[
  \int_0^1\frac{\ln\!\big(1-z\,x(1-x)^3\big)}{x(1-x)}\,dx
  =\frac12\sum_{k=1}^4\ln^2\!\Big(1-\frac1{\rho_k}\Big),
\]
where $\rho_1,\dots,\rho_4$ are the roots of the fiber polynomial \eqref{eq:fiber}.
\end{lemma}

\begin{proof}
For $0<\varepsilon<1$ set
\[
  I_\varepsilon=\int_0^{1-\varepsilon}\frac{\ln\!\big(1-z\,x(1-x)^3\big)}{x(1-x)}\,dx .
\]
On $[0,1-\varepsilon]$ the integrand is bounded, so the factorisation
$\ln(1-z\,x(1-x)^3)=\sum_k\ln(1-x/\rho_k)$ and the splitting
$\tfrac1{x(1-x)}=\tfrac1x+\tfrac1{1-x}$ may be performed term by term, each resulting integral being
finite:
\[
  I_\varepsilon=\sum_{k=1}^4\left[\int_0^{1-\varepsilon}\frac{\ln(1-x/\rho_k)}{x}\,dx
  +\int_0^{1-\varepsilon}\frac{\ln(1-x/\rho_k)}{1-x}\,dx\right].
\]
As $\varepsilon\to0^+$ the first integral converges to $-\Li_2(1/\rho_k)$. The second diverges
individually; integrating by parts and expanding,
\[
  \int_0^{1-\varepsilon}\frac{\ln(1-x/\rho_k)}{1-x}\,dx
  =-\ln\varepsilon\cdot\ln\!\Big(1-\frac1{\rho_k}\Big)-\Li_2\!\Big(\frac{1}{1-\rho_k}\Big)+o(1).
\]
Summing over $k$, the divergent contributions carry the total coefficient
$\sum_{k}\ln(1-1/\rho_k)$, which vanishes. Indeed $\prod_k(1-1/\rho_k)=p(1)/p(0)=1$, $p$ being the
monic fiber polynomial \eqref{eq:fiber}; and passing from $\sum_k\ln w_k$ to $\ln\prod_k w_k$ with
principal logarithms is legitimate here because $\sum_{k=1}^4\arg(1-1/\rho_k)=0$: the nonreal roots
form a conjugate pair, so $\arg(1-1/\rho)+\arg(1-1/\bar\rho)=0$, while by Remark~\ref{rem:signs}
$1-1/a>0$ and $1-1/b>0$ each contribute argument $0$. No multiple of $2\pi i$ intervenes, so
\[
  \sum_{k=1}^4\ln\Big(1-\frac1{\rho_k}\Big)=\ln 1=0 .
\]
Therefore $I_\varepsilon$ converges as $\varepsilon\to0^+$, with
\[
  \lim_{\varepsilon\to0^+}I_\varepsilon
  =-\sum_{k=1}^4\left[\Li_2\!\Big(\frac1{\rho_k}\Big)+\Li_2\!\Big(\frac{1}{1-\rho_k}\Big)\right].
\]
Finally, putting $w=1/\rho_k$ one has $\tfrac{w}{w-1}=\tfrac{1}{1-\rho_k}$, so Landen's identity
$\Li_2(w)+\Li_2\!\big(\tfrac{w}{w-1}\big)=-\tfrac12\ln^2(1-w)$ collapses each bracket to
$-\tfrac12\ln^2(1-1/\rho_k)$, which is the stated result. \qedhere
\end{proof}

Combining Lemmas~\ref{lem:factor} and~\ref{lem:landen} with \eqref{eq:sym-weight},
\begin{equation}\label{eq:landen-form}
  \boxed{\;\mathcal{I}(z)=\frac38\sum_{k=1}^4\ln^2\!\Big(1-\frac1{\rho_k}\Big).\;}
\end{equation}
No dilogarithm and no five-term relation appear; the only algebraic arguments are $1-1/\rho_k$.

\begin{remark}
The reduction \eqref{eq:landen-form} is not new: D'Aurizio and Di Trani~\cite{DaurizioDiTrani} obtain
the corresponding identity for the mirror base $x^3(1-x)$, by a different route (they use the
substitution $x\mapsto1-x$, available because of the shape of their base, in place of the
Beta-quotient of Lemma~\ref{lem:factor}). Their displayed constant is $-\tfrac32$; the correct value
for their normalisation is $-\tfrac38$, as one checks numerically at any convenient $z$ --- for
instance at $z=-2$, where the series equals $-0.4678479\ldots$ while the sum of squared logarithms is
$1.2475944\ldots$, in ratio $-3/8$. Their cubic-base constant $-\tfrac13$ is correct. Lemma~\ref{lem:factor}
gives \eqref{eq:landen-form} with the factor $\tfrac{q}{p+q}\cdot\tfrac12=\tfrac38$ for every
$(p,q)$, and in particular fixes the quartic constant.
\end{remark}

\subsection{Collapse to three generators}

Split the four roots into the real pair $a,b$ and the conjugate pair $\rho,\bar\rho=re^{\pm i\theta}$.
We distinguish the \emph{root angle} $\arg\rho=\theta$ from the \emph{logarithmic angle}
$\arg(1-1/\rho)$; it is the latter that enters the squared logarithm, and Lemma~\ref{lem:angle}
relates the two.
Since $\ln^2(1-1/\rho)+\ln^2(1-1/\bar\rho)=2\,\re\,\ln^2(1-1/\rho)
=2\big[\ln^2|1-1/\rho|-\arg^2(1-1/\rho)\big]$,
\begin{equation}\label{eq:split}
  \mathcal{I}(z)=\frac38\Big[\underbrace{\ln^2\!\big|1-\tfrac1a\big|+\ln^2\!\big|1-\tfrac1b\big|
      +2\ln^2\!\big|1-\tfrac1\rho\big|}_{\text{logarithmic part}}
  \;-\;\underbrace{2\arg^2\!\big(1-\tfrac1\rho\big)}_{\pi^2\text{ part}}\Big].
\end{equation}

\begin{lemma}\label{lem:collapse}
Let $L_a=\ln|a|$, $L_b=\ln|b|$, $L_r=\ln r$. Then
\[
  \ln\!\big|1-\tfrac1a\big|=\tfrac13(-3L_a+L_b+2L_r),\quad
  \ln\!\big|1-\tfrac1b\big|=\tfrac13(L_a-3L_b+2L_r),\quad
  \ln\!\big|1-\tfrac1\rho\big|=\tfrac13(L_a+L_b-2L_r).
\]
\end{lemma}

\begin{proof}
Every root satisfies the fiber relation $\rho(1-\rho)^3=\tfrac1z$; taking moduli and logarithms,
$\ln|\rho|+3\ln|1-\rho|=\ln|1/z|$. The product of all four roots equals the constant term
$1/z$ of \eqref{eq:fiber}, so $\ln|1/z|=L_a+L_b+2L_r$ (the pair contributes $\ln r^2=2L_r$).
Applying $\ln|\rho|+3\ln|1-\rho|=L_a+L_b+2L_r$ to each of $a$, $b$ and $\rho$, together with
$\ln|1-1/\rho|=\ln|1-\rho|-\ln|\rho|$, gives the three stated identities after elimination. \qedhere
\end{proof}

Substituting Lemma~\ref{lem:collapse} into the logarithmic part of \eqref{eq:split} yields, by a
direct expansion,
\begin{equation}\label{eq:Q}
  \frac38\Big[\big(\tfrac{-3L_a+L_b+2L_r}{3}\big)^2+\big(\tfrac{L_a-3L_b+2L_r}{3}\big)^2
  +2\big(\tfrac{L_a+L_b-2L_r}{3}\big)^2\Big]
  = Q(L_a,L_b,L_r),
\end{equation}
where
\begin{equation}\label{eq:Qdef}
  Q(L_a,L_b,L_r)=\tfrac12L_a^2+\tfrac12L_b^2+\tfrac23L_r^2
  -\tfrac13L_aL_b-\tfrac23L_aL_r-\tfrac23L_bL_r .
\end{equation}
(The identity \eqref{eq:Q}=\eqref{eq:Qdef} is a polynomial identity in $L_a,L_b,L_r$, verified by
expansion.) It remains to identify the $\pi^2$ part.

\subsection{The $\pi^2$ coefficient}

\begin{lemma}\label{lem:angle}
For an admissible member with forced angle $\theta=j\pi/N$ and $z<0$,
\[
  \arg\!\big(1-\tfrac1\rho\big)=\tfrac23(\pi-2\theta),
  \qquad\text{hence}\qquad
  -\tfrac38\cdot 2\arg^2\!\big(1-\tfrac1\rho\big)=-\tfrac13\Big(1-\tfrac{2j}{N}\Big)^2\pi^2 .
\]
\end{lemma}

\begin{proof}
Write $\rho=re^{i\theta}$ with $0<\theta<\pi$, so that $\rho$ lies in the open upper half-plane. Then
$1-\rho=1-r\cos\theta-i\,r\sin\theta$ has strictly negative imaginary part (as $r\sin\theta>0$), so it
lies in the open lower half-plane and $-\pi<\arg(1-\rho)<0$. Hence
$\arg\rho+3\arg(1-\rho)\in(\theta-3\pi,\theta)\subset(-3\pi,\pi)$. Since $z<0$ the product
$1/z=\rho(1-\rho)^3$ is a negative real number, so its argument is an odd multiple of $\pi$; of the
odd multiples in $(-3\pi,\pi)$ only $-\pi$ is attainable, since $+\pi$ would force
$3\arg(1-\rho)=\pi-\theta>0$, contradicting $\arg(1-\rho)<0$. Therefore
$\arg\rho+3\arg(1-\rho)=-\pi$, giving
\begin{equation}\label{eq:trisect}
  \arg(1-\rho)=-\frac{\pi+\theta}{3}:
\end{equation}
the cubic factor $(1-x)^3$ \emph{trisects} the deficit angle $\pi+\theta$. Now write
$1-\tfrac1\rho=\tfrac{\rho-1}{\rho}=-\tfrac{1-\rho}{\rho}$, so that
\[
  \arg\!\Big(1-\frac1\rho\Big)
  =\arg\!\big(-(1-\rho)\big)-\arg(\rho)
  =\big[\arg(1-\rho)+\pi\big]-\theta .
\]
Substituting \eqref{eq:trisect},
\[
  \arg\!\Big(1-\frac1\rho\Big)
  =-\frac{\pi+\theta}{3}+\pi-\theta
  =\frac{2\pi}{3}-\frac{4\theta}{3}
  =\frac23(\pi-2\theta).
\]
Finally, with $\theta=j\pi/N$,
\[
  -\frac34\arg^2\!\Big(1-\frac1\rho\Big)
  =-\frac34\cdot\frac49(\pi-2\theta)^2
  =-\frac13(\pi-2\theta)^2
  =-\frac13\Big(1-\frac{2j}{N}\Big)^2\pi^2 .
\]
\qedhere
\end{proof}

\begin{remark}
The coefficient $-\tfrac13$ arises from two independent occurrences of the exponent $3$ of $(1-x)^3$:
the trisection denominator in \eqref{eq:trisect}, and the factor
$\tfrac38=\tfrac{q}{p+q}\cdot\tfrac12$ of \eqref{eq:landen-form}.
\end{remark}

The same angle information determines the modulus $r$, so that the matching cubic
\eqref{eq:matching} need not be solved at all.

\begin{lemma}\label{lem:modulus}
For an admissible member with forced angle $\theta$ one has $\theta<\pi/2$ and
\[
  r\;=\;\frac{\sin\!\big(\tfrac{\pi+\theta}{3}\big)}{\sin\!\big(\tfrac{\pi+4\theta}{3}\big)} .
\]
\end{lemma}

\begin{proof}
Since $\rho$ is not real, the points $0,1,\rho$ are the vertices of a genuine triangle. The angle at
$0$ is $\arg\rho=\theta$. At the vertex $1$ the ray towards $0$ has direction $-1$, of argument $\pi$,
while the ray towards $\rho$ has direction $\rho-1$, of argument $\pi+\arg(1-\rho)$; by
\eqref{eq:trisect} the angle between them is therefore $\varphi:=\tfrac{\pi+\theta}{3}$. The angle at
$\rho$ is the remaining $\pi-\theta-\varphi$, and positivity of this angle gives
$\theta+\varphi=\tfrac{\pi+4\theta}{3}<\pi$, i.e.\ $\theta<\pi/2$. The law of sines applied to the
side $[0,1]$ of length $1$, opposite the vertex $\rho$, and the side $[0,\rho]$ of length $r$,
opposite the vertex $1$, gives
\[
  \frac{r}{\sin\varphi}=\frac{1}{\sin(\theta+\varphi)},
\]
which is the stated formula. \qedhere
\end{proof}

\begin{remark}
Lemma~\ref{lem:modulus} removes the ambiguity noted after \eqref{eq:matching}: the matching cubic may
have several positive roots, but only the value above is admissible, and it is given in closed form
by the angle alone. It also explains the pentagon's exact value $r=1$, since
$\sin\tfrac{2\pi}{5}=\sin\tfrac{3\pi}{5}$, and shows that no admissible member has $\theta\ge\pi/2$;
in particular $2j<N$ and the coefficient $c=-\tfrac13(1-2j/N)^2$ never vanishes.
\end{remark}

Assembling \eqref{eq:split}, \eqref{eq:Q}--\eqref{eq:Qdef} and Lemma~\ref{lem:angle}:

\begin{theorem}\label{thm:main}
Let $\theta=j\pi/N$ with $\gcd(j,N)=1$, and let $r$ be a positive real root of
$(\kappa^3-2\kappa)r^3+3(1-\kappa^2)r^2+3\kappa r-1=0$ with $\kappa=2\cos\theta$; for an admissible
member this is the value $r=\sin\big(\tfrac{\pi+\theta}{3}\big)/\sin\big(\tfrac{\pi+4\theta}{3}\big)$
of Lemma~\textup{\ref{lem:modulus}}. Put
$S=3-\kappa r$, $P=3-r^2-3\kappa r+\kappa^2r^2$, $z=1/(Pr^2)$ and
$a,b=\tfrac12\big(S\pm\sqrt{S^2-4P}\big)$, with $a$ the larger root. If $S^2-4P>0$ and $|z|<256/27$,
then necessarily $z<0$ and $a>0>b$ (Remark~\textup{\ref{rem:signs}}), the remaining two roots of
\eqref{eq:fiber} are $re^{\pm i\theta}$, and
\begin{align*}
  -\frac{z}{4}\;{}_5F_4\!\left(\begin{matrix}1,1,1,\tfrac43,\tfrac53\\[2pt]
      \tfrac54,\tfrac32,\tfrac74,2\end{matrix}\;\middle|\;\frac{27z}{256}\right)
  &= -\frac13\Big(1-\frac{2j}{N}\Big)^2\pi^2
    + \tfrac12\ln^2 a+\tfrac12\ln^2|b|+\tfrac23\ln^2 r\\[2pt]
  &\qquad-\tfrac13\ln a\ln|b|-\tfrac23\ln r\,(\ln a+\ln|b|).
\end{align*}
\end{theorem}

\begin{remark}[the signs of $z$, $a$, $b$ are automatic]\label{rem:signs}
An admissible member has $z<0$. Indeed $x(1-x)^3\le 27/256$ on $[0,1]$ and is negative outside it, so
a real root of $x(1-x)^3=1/z$ with $z>0$ requires $1/z\le27/256$, i.e.\ $z\ge z_c=256/27$. Hence any
member with real $a,b$ and $|z|<z_c$ has $z<0$, which is the hypothesis used in Lemma~\ref{lem:angle}.
The cubic in $r$ may have more than one positive root; only a branch with $S^2-4P>0$ and $|z|<z_c$ is
admissible, and it is this branch that is meant throughout.

The signs of $a$ and $b$ are likewise automatic, so that $\ln a$ and $\ln|b|$ in
Theorem~\ref{thm:main} are the real logarithms of positive reals. Indeed the product of the four
roots equals the constant term $1/z$ of \eqref{eq:fiber}, and the conjugate pair contributes
$|\rho|^2=r^2>0$; hence $ab\,r^2=1/z<0$, so $ab<0$ and $a,b$ have opposite signs. With $a$ the larger
root this gives
\[
  a>0>b ,
\]
and all logarithms occurring in Theorem~\ref{thm:main} are taken as real logarithms of the positive
quantities $a$, $|b|$ and $r$; no complex branch of $\log$ is needed. Since $a>1$ and $b<0$ we have
in addition
\[
  1-\frac1a>0,\qquad 1-\frac1b>0,
\]
which is used in the proof of Lemma~\ref{lem:landen} to control principal arguments.
\end{remark}

\begin{remark}
The coefficient of $\pi^2$ depends only on the angle index and equals
$-\tfrac13(N-2j)^2/N^2$. The logarithmic form $Q$ is universal. The root-difference logarithms
$\ln|\rho_i-\rho_j|$, which would appear in a five-term reduction and do \emph{not} lie in
$\mathrm{span}_\Q\{\ln a,\ln|b|,\ln r\}$, never occur: the proof uses only Landen's two-term
identity and the multiplicative fiber relation. This is why exactly three logarithmic generators
suffice.
\end{remark}

\section{Four evaluations}\label{sec:examples}

Table~\ref{tab:main} records the admissible members at $\theta=\pi/5,\pi/4,2\pi/7,2\pi/9$. Each
value in the last row is the number $\mathcal{I}(z)$ obtained by summing the series
\eqref{eq:series}. The computations were performed in $40$-digit floating-point arithmetic: the
defining algebraic data $r,z,a,b$ were obtained as roots of the exact polynomials above, and
$\mathcal{I}(z)$ by direct summation of \eqref{eq:series}, whose terms decay like
$(z/z_c)^n n^{-5/2}$, to $600$ terms; the closed form of Theorem~\ref{thm:main} was then evaluated
independently. Writing $\mathcal{F}$ for the closed form of Theorem~\ref{thm:main}, the residuals
$|\mathcal{I}(z)-\mathcal{F}|$ were $1.9\cdot10^{-60}$, $2.3\cdot10^{-61}$, $1.9\cdot10^{-61}$ and
$1.0\cdot10^{-60}$ for the pentagon, octagon, heptagon and nonagon respectively --- in each case at
the level of the working precision. Every column is reproducible from $(j,N)$ directly: take $r$ from
Lemma~\ref{lem:modulus}, form $S,P$ and $z$ by \eqref{eq:matching}--\eqref{eq:z}, and solve the
resulting quadratic for $a,b$.

\begin{table}[H]
\centering
\renewcommand{\arraystretch}{1.3}
\begin{tabular}{lcccc}
\toprule
 & pentagon & octagon & heptagon & nonagon \\
\midrule
$\theta=j\pi/N$        & $\pi/5$ & $\pi/4$ & $2\pi/7$ & $2\pi/9$ \\
$\kappa=2\cos\theta$   & $\tfrac{1+\sqrt5}{2}$ & $\sqrt2$ & $2\cos\tfrac{2\pi}{7}$ & $2\cos\tfrac{2\pi}{9}$ \\
$r$                    & $1$ & $1.11535507$ & $1.24697960$ & $1.04331595$ \\
$z=1/\big(a(1-a)^3\big)$ & $-4.23606798$ & $-1.64711432$ & $-0.80193774$ & $-2.79393554$ \\
$a$ (larger real root) & $1.53568739$ & $1.70832885$ & $1.87316173$ & $1.60625629$ \\
$b$ (smaller real root)& $-0.15372138$ & $-0.28567911$ & $-0.42811986$ & $-0.20470907$ \\
$c=-\tfrac13(1-\tfrac{2j}{N})^2$ & $-\tfrac{3}{25}$ & $-\tfrac1{12}$ & $-\tfrac3{49}$ & $-\tfrac{25}{243}$ \\
$\mathcal{I}(z)$       & $0.92876942$ & $0.38958726$ & $0.19498950$ & $0.63806877$ \\
\bottomrule
\end{tabular}
\caption{Four admissible members of the ${}_5F_4$ family. The pentagon has $r=1$ exactly, so the
forced pair lies on the true unit circle at $e^{\pm i\pi/5}$; the three $\ln r$--terms vanish and the
closed form has three logarithmic terms. In general the quadratic form involves at most
$\binom{d+1}{2}$ distinct monomials, $d$ being the number of independent generators used. Vieta's
formula gives the exact relation $a\,|b|\,r^2=-1/z$ for \emph{every} member. The pentagon has $r=1$
exactly, and the heptagon satisfies $z=-1/r$, whence $a\,|b|\,r=1$; in both cases two generators
suffice. For the octagon and nonagon we know of no such relation and write the formulas with three
generators, but we do not claim multiplicative independence.}
\label{tab:main}
\end{table}

\begin{example}[pentagon, $\Q(\sqrt5)$]
Here $\kappa=\varphi=\tfrac{1+\sqrt5}{2}$ and the matching cubic \eqref{eq:matching} has the
\emph{exact} root $r=1$: the forced pair lies on the true unit circle at $e^{\pm i\pi/5}$. Then
$z=-(2+\sqrt5)=-\varphi^3$, with minimal polynomial $z^2+4z-1=0$, so the ${}_5F_4$ argument is
\[
  \frac{27z}{256}=-\frac{27(2+\sqrt5)}{256}\in\Q(\sqrt5).
\]
Since $\ln r=0$, Theorem~\ref{thm:main} reduces to the three-term identity
\[
  -\frac{z}{4}\;{}_5F_4\!\left(\begin{matrix}1,1,1,\tfrac43,\tfrac53\\[2pt]
      \tfrac54,\tfrac32,\tfrac74,2\end{matrix}\;\middle|\;-\frac{27(2+\sqrt5)}{256}\right)
  = -\frac{3\pi^2}{25}+\tfrac12\ln^2 a+\tfrac12\ln^2|b|-\tfrac13\ln a\ln|b|,
\]
Here $a,b$ are the roots of the real quadratic factor of the fiber \eqref{eq:fiber}, namely
\[
  2x^2-(5-\sqrt5)\,x+2(2-\sqrt5)=0 ,
\]
the fiber itself being $x^4-3x^3+3x^2-x+(2-\sqrt5)=0$ since $1/z=2-\sqrt5$.
\end{example}

\begin{example}[octagon, $\Q(\sqrt2)$ and $\Q(\sqrt3)$]
Here $\kappa=\sqrt2$, and the leading coefficient $\kappa^3-2\kappa$ of the matching cubic
\eqref{eq:matching} vanishes, so it degenerates to the quadratic $3r^2-3\sqrt2\,r+1=0$ with admissible
(larger) root
\[
  r=\frac{\sqrt2}{2}+\frac{\sqrt6}{6}=\sqrt{\tfrac23+\tfrac{\sqrt3}{3}}\;=\;1.11535507\ldots
\]
The smaller root $\tfrac{\sqrt2}{2}-\tfrac{\sqrt6}{6}$ gives $S^2-4P<0$ and is inadmissible. One finds
$z=\tfrac94(1-\sqrt3)=-1.64711432\ldots$, with minimal polynomial $8z^2-36z-81=0$; thus
$z\in\Q(\sqrt3)$ and the ${}_5F_4$ argument is
\[
  \frac{27z}{256}=\frac{243(1-\sqrt3)}{1024}\in\Q(\sqrt3),
\]
while the forced angle lives in $\Q(\sqrt2)$. This is the six-term member of Table~\ref{tab:main}.
With $a=1.70832885\ldots$, $|b|=0.28567911\ldots$, $r=1.11535507\ldots$, the full evaluation reads
\begin{align*}
  -\frac{z}{4}\;{}_5F_4\!\left(\begin{matrix}1,1,1,\tfrac43,\tfrac53\\[2pt]
      \tfrac54,\tfrac32,\tfrac74,2\end{matrix}\;\middle|\;\frac{243(1-\sqrt3)}{1024}\right)
  &= -\frac{\pi^2}{12}
    + \tfrac12\ln^2 a+\tfrac12\ln^2|b|+\tfrac23\ln^2 r\\[2pt]
  &\qquad-\tfrac13\ln a\ln|b|-\tfrac23\ln r\,(\ln a+\ln|b|)
    \;=\;0.38958726\ldots,
\end{align*}
all six terms of $Q$ being present; we know of no multiplicative relation among $a,|b|,r$ here
beyond the universal $a|b|r^2=-1/z$, and make no claim of independence.
Here $a,b$ are the roots of the real quadratic factor of the fiber \eqref{eq:fiber}, namely
\[
  3x^2-(6-\sqrt3)\,x+2(1-\sqrt3)=0 ,
\]
the fiber itself being $x^4-3x^3+3x^2-x-\tfrac29(1+\sqrt3)=0$ since $1/z=-\tfrac29(1+\sqrt3)$.
\end{example}

\begin{remark}[rational arguments]
Four conditions must be distinguished: rationality of $\cos\theta$, of $\kappa=2\cos\theta$, of $r$,
and of $z$ --- the last being equivalent to rationality of the ${}_5F_4$ argument $27z/256$, since the
two differ by the rational factor $27/256$. By Niven's theorem the only rational values of
$\cos\theta$ at rational multiples of $\pi$ are $0,\pm\tfrac12,\pm1$, so $\kappa\in\Q$ exactly for
$\theta\in\{\pi/3,\pi/2,2\pi/3\}$ in the relevant range. Since $r$ and $z$ are obtained from $\kappa$
by \eqref{eq:matching}--\eqref{eq:z}, rationality of $z$ requires $\kappa\in\Q$ but is not implied by
it; among the three Niven angles only $\theta=\pi/2$ in fact yields rational $r$ and $z$. Of them only $\theta=\pi/2$ yields a rational $z$ (namely $z=9/8$, argument
$243/2048$), but there $c=-\tfrac13(1-2\cdot\tfrac12)^2=0$ and moreover $S^2-4P<0$, so the member is
not admissible: the ``real pair'' $a,b$ is itself complex, and the value is not of pure
$\pi^2+\log^2$ type. Thus a rational argument and a nonzero $\pi^2$-coefficient never coincide in this
family; the quadratic-field pentagon and octagon are the closest admissible members, their arguments
being quadratic irrationals.
\end{remark}

\begin{remark}
Among the admissible members we have computed (all $\theta=j\pi/N$ with $\gcd(j,N)=1$, $N\le25$),
the pentagon and octagon are the only two whose defining datum lies in a \emph{quadratic} field:
$z\in\Q(\sqrt5)$ for the pentagon (indeed $r=1$), and $z\in\Q(\sqrt3)$ for the octagon; every other
computed member has $z$ of degree $\ge3$. This fits the general shape of the construction: $z$ lies
in an extension of $\Q(\kappa)$ generated by a root of the matching cubic, and
$[\Q(\kappa):\Q]=\varphi(2N)/2$, so the degree can drop only when that cubic degenerates or factors
--- as it does at $\theta=\pi/4$, where its leading coefficient vanishes, and at $\theta=\pi/5$,
where $r=1$ is a root. We do not attempt to classify all angles at which such a degeneration
occurs.
\end{remark}

\subsection{No Clausen constant occurs}\label{sec:clausen}

Since $\mathrm{Cl}_2(\psi)=\sum_{k\ge1}k^{-2}\sin k\psi$ is the imaginary part of $\Li_2(e^{i\psi})$,
one might ask whether a Clausen value --- Catalan's constant $G=\mathrm{Cl}_2(\pi/2)$, say --- can
appear in an evaluation of $\mathcal{I}(z)$. It cannot: $\mathcal{I}(z)$ is a real integral, so its
nonreal roots occur in conjugate pairs, and since $\mathrm{Cl}_2(-\psi)=-\mathrm{Cl}_2(\psi)$ the
contributions of $\rho$ and $\bar\rho$ cancel. The cancellation is insensitive to the angle: a pair
placed exactly at $e^{\pm i\pi/2}$ contributes
$\Li_2(i)+\Li_2(-i)=\tfrac12\Li_2(-1)=-\tfrac{\pi^2}{24}$, the two copies of $\pm G$ having cancelled.
Two facts should be kept apart. The cancellation of the Clausen terms uses only conjugate pairing
and holds whatever the angle. What survives is $\arg^2(1-1/\rho)$, and it is the separate
hypothesis that the angle be a rational multiple of $\pi$ that makes this a rational multiple of
$\pi^2$. For the quintic bases of Section~\ref{sec:boundary} the Clausen terms still cancel, but the
surviving $\arg^2$ sits at an irrational angle; it is this, not a Clausen constant, that obstructs
purity.

\section{The maximal order: a real-root count}\label{sec:boundary}

\begin{definition}\label{def:pure}
An evaluation of $\mathcal{I}_{p,q}(z)$ is called \emph{pure} if it can be written as
\[
  c\,\pi^2+\sum_{i\le j}c_{ij}\log\alpha_i\log\alpha_j,
  \qquad c,c_{ij}\in\Q,
\]
for finitely many positive real algebraic numbers $\alpha_i$. Theorem~\ref{thm:main} asserts that
every admissible member of the ${}_5F_4$ family is pure, with $\{\alpha_i\}=\{a,|b|,r\}$.
\end{definition}

Whether the conjugate-pair construction yields a \emph{convergent} family of \emph{pure}
$\pi^2+\log^2$ evaluations is decided by a single count. We work in the regime $z<0$, where all
admissible members of Theorem~\ref{thm:main} lie.

\begin{lemma}\label{lem:realcount}
Let $\nu=\nu(p,q)\in\{0,1,2\}$ be the number of odd elements of $\{p,q\}$. For $z<0$ the fiber
$x^p(1-x)^q=1/z$ has exactly $\nu$ real roots, and consequently $(s-\nu)/2$ complex-conjugate pairs,
where $s=p+q$.
\end{lemma}

\begin{proof}
Put $w=1/z<0$ and $f(x)=x^p(1-x)^q$. On $(0,1)$ both factors are positive, so $f>0$ and
$f(x)=w<0$ has no root there. On $(-\infty,0)$ one has $\operatorname{sign}f=(-1)^p$, so $f<0$
throughout iff $p$ is odd; and there $f'=x^{p-1}(1-x)^{q-1}(p-sx)$ has constant sign
$(-1)^{p-1}$, so $f$ is monotonic from $f(-\infty)=-\infty$ to $f(0)=0$ and meets the level $w<0$
exactly once. Symmetrically, on $(1,\infty)$ one has $\operatorname{sign}f=(-1)^q$, and when $q$ is
odd $f$ is monotonic from $f(1)=0$ to $f(\infty)=-\infty$, meeting $w$ once. Each odd exponent thus
contributes exactly one real root and each even exponent none, so the number of real roots is $\nu$;
the rest are complex and paired.
\end{proof}

The conjugate-pair construction forces \emph{one} complex pair to a prescribed rational angle,
consuming the single free parameter $z$. Purity requires \emph{every} complex pair to be at a
rational angle. The number of pairs \emph{not} forced is therefore
\begin{equation}\label{eq:leftover}
  \ell(p,q)\;=\;\frac{s-\nu}{2}\;-\;1 .
\end{equation}
When $\ell=0$ the forced pair is the only complex pair, and the member is automatically pure: this is
the situation of Theorem~\ref{thm:main}. When $\ell\ge1$ at least one complex pair is left unforced;
for such a member to be pure, that pair would have to sit at a rational angle by coincidence, with no
parameter available to arrange it. The following theorem records the exact $\ell=0$ dichotomy, which
is what governs the existence of a \emph{one-parameter} pure family; the impurity of the $\ell\ge1$
members is addressed separately below.

\begin{theorem}[maximality of automatic purity for the one-pair construction]\label{thm:boundary}
Let $s=p+q$ and let $\ell(p,q)$ be as in \eqref{eq:leftover}. Then
\[
  \ell(p,q)=0 \iff s-\nu=2 ,
\]
where $\nu$ is as in Lemma~\ref{lem:realcount}; this holds precisely when $s=3$ or $s=4$; up to interchanging $p$ and $q$ these are the cases
$(p,q)=(1,2)$ and $(1,3)$, i.e.\ the ${}_4F_3$ base $x(1-x)^2$ and the ${}_5F_4$ base $x(1-x)^3$.
Consequently:
\begin{enumerate}[label=\textup{(\roman*)}]
\item For $s\le4$ with $\ell=0$, the forced conjugate pair is the \emph{only} complex pair of the
fiber for $z<0$. Every admissible member is then automatically pure, and the construction of
Theorem~\ref{thm:main} yields a convergent one-parameter family of pure $\pi^2+\log^2$ evaluations.
\item For every $s\ge5$ one has $\ell\ge1$: after the forced pair, at least one further conjugate
pair remains, and no parameter is left with which to constrain its angle. Hence no convergent
one-parameter family of pure evaluations is produced by this construction.
\end{enumerate}
Thus $s=4$ is the maximal order --- and ${}_5F_4$ the largest central-binomial series --- for which
this one-pair construction \emph{automatically forces} purity in the convergent regime $z<0$.
\end{theorem}

\begin{proof}
By Lemma~\ref{lem:realcount} and \eqref{eq:leftover}, $\ell=0$ iff
$s-\nu=2$. Since $\nu\le2$, this forces $s\le4$; and $s=3$ (with $p+q$ odd, so exactly one of $p,q$
is odd and $\nu=1$) gives $3-1=2$, while $s=4$ with $p=1,q=3$ (both odd, $\nu=2$) gives $4-2=2$. For
$s\ge5$ one has $s-\nu\ge5-2=3>2$, so $\ell\ge1$ and an unforced pair survives.
Convergence of the resulting members ($|z|<z_c$) for $s=3,4$ is exhibited in
Section~\ref{sec:examples} and below.
\end{proof}

\begin{remark}
Since $1/z$ is real, every root satisfies $p\arg\rho+q\arg(1-\rho)\equiv0\pmod\pi$, so $\arg\rho$ is a
rational multiple of $\pi$ if and only if $\arg(1-\rho)$ is. An unforced pair sitting at a rational
angle would therefore make $0,1,\rho$ a triangle all of whose angles are rational multiples of $\pi$,
with modulus a ratio of sines of such angles, exactly as in Lemma~\ref{lem:modulus}. Deciding whether
two such configurations can occur on a single fiber is a question about vanishing sums of roots of
unity in the sense of Mann and Conway--Jones; we do not pursue it here.
\end{remark}

\begin{remark}
Part (ii) concerns the construction, not the individual members: for $s\ge5$ the forced pair no
longer exhausts the complex roots, so purity ceases to be automatic. An isolated parameter value at
which the unforced pair happens to lie at a rational angle is not excluded by the count; whether any
such value exists we do not know, and we make no claim about it here.
\end{remark}

\begin{remark}
The distinction between the two admissible orders is qualitative. For $s=3$ (${}_4F_3$) the fiber for
$z<0$ has one real root and one complex pair; forcing the pair to a rational angle gives the
$\binom{3n}{n}$ evaluations, convergent for $|z|<27/4$. For $s=4$ (${}_5F_4$) the fiber has two real
roots and one complex pair, giving the family of the present paper. The earlier all-real
configuration of~\cite{DaurizioDiTrani} (three real roots for the cubic base) occurs only for $z>0$
with $z>z_c$ and is non-convergent; the convergent pure families live at $z<0$, and exist for both
$s=3$ and $s=4$.
\end{remark}

\begin{remark}[computational evidence, $s=5$]
For all three quintic shapes $x(1-x)^4$, $x^2(1-x)^3$, $x^3(1-x)^2$ one has $s=5$ and
$\nu=1$, so $\ell=1$: each convergent forced member retains exactly one
unforced complex pair. Whether such a pair can land at a rational angle is not settled by
Theorem~\ref{thm:boundary}. We tested it by an exhaustive search over $\theta=j\pi/N$ with
$\gcd(j,N)=1$ and $N\le50$, comprising $3{,}641$ members. Within that range no quintic fiber has two
complex pairs both at rational angles, and no convergent member has a totally real cofactor.

Three configurations illustrate how the two possible routes to purity fail.
\begin{itemize}
\item \textbf{An irrational leftover pair.} Forcing $\theta=\pi/3$ in $x(1-x)^4$ gives the convergent
point $z=5.7756\ldots<z_c=12.207\ldots$, but the fiber carries a second complex pair at the
irrational angle $0.16314\pi$.
\item \textbf{Divergence.} Forcing $\theta=\pi/7$ in the same base places the sole
totally-real-cofactor branch at $z=12.36\ldots>z_c$, where \eqref{eq:series} diverges. The remaining
branches, including the one at $|z|=31.79\ldots$, keep a complex pair at an irrational angle and are
impure rather than divergent.
\item \textbf{A rational pair that does not help.} One may force a pair \emph{exactly} at
$\theta=\pi/2$ by imposing a factor $x^2+1$: for $x^3(1-x)^2$ this succeeds uniquely at $z=-2$, with
fiber $(x^2+1)(x^3-2x^2+2)$. The quadratic supplies the clean pair $\pm i$, but the cubic cofactor
contributes a complex pair at the irrational angle $0.1285\pi$, so the member is impure.
\end{itemize}
In each case the degree-three cofactor is what obstructs purity, and in every case tested it failed
to be totally real within $|z|<z_c$. At the other end the situation is transparent: for $s=3$, three
real roots of $x(1-x)^2=1/z$ require $0<1/z\le4/27$, that is $z\ge z_c=27/4$, so the totally real
configuration is non-convergent. This is why the ${}_4F_3$ pure family, like the ${}_5F_4$, lives at
$z<0$ with a single complex pair.

We stress that all of this is evidence, not proof: the search covers only $N\le50$ and establishes
nothing for larger $N$, nor for $s\ge6$, and none of the computational statements in this remark is
used in the proof of Theorem~\ref{thm:boundary} or of any other result above. The search ran over all pairs $(j,N)$ with $2\le N\le50$ and $1\le j<N$, $\gcd(j,N)=1$, one
representative per angle (so that $j/N$ and $j'/N'$ with $j/N=j'/N'$ are not counted twice); for each
$(j,N)$ and each of the three quintic shapes, every positive real root of the corresponding matching
equation was taken, giving $3{,}641$ members in all. Fiber polynomials were formed exactly over the
relevant number field and their roots located to $40$ significant digits. An angle was recorded as
rational when it agreed with some $j\pi/N$, $N\le50$, to within $10^{-25}$; this is a heuristic
detection criterion, not a proof of irrationality for the angles it rejects.
\end{remark}

\begin{remark}
For the quintic bases the symmetric-weight identity still gives an exact value: the analogue of
\eqref{eq:landen-form} reads
$\mathcal{I}_{p,q}(z)=\tfrac{q}{p+q}\cdot\tfrac12\sum_\rho\ln^2(1-1/\rho)$,
and a reflection-symmetric combination of mirror shapes reduces the transcendental content to
algebraic $\arg^2$ terms. But these are not rational multiples of $\pi^2$ unless every complex root
sits at a rational angle, which by Lemma~\ref{lem:realcount} fails for a single quintic shape. (The
torsion criterion for such Bloch elements is that of Nahm~\cite{Nahm,Zagier}.)
\end{remark}

\end{document}